\documentclass[12pt,oneside,english]{amsart}
\usepackage[T1]{fontenc}
\usepackage[latin9]{inputenc}
\usepackage{geometry}
\usepackage{amsthm}
\usepackage{amstext}

\makeatletter
\numberwithin{equation}{section}
\numberwithin{figure}{section}
\theoremstyle{plain}
\newtheorem{thm}{\protect\theoremname}
  \theoremstyle{plain}
  \newtheorem{prop}[thm]{\protect\propositionname}
  \theoremstyle{plain}
  \newtheorem{lem}[thm]{\protect\lemmaname}

\makeatother

\usepackage{babel}
  \providecommand{\lemmaname}{Lemma}
  \providecommand{\propositionname}{Proposition}
\providecommand{\theoremname}{Theorem}

\begin{document}

\title{The Median Largest Prime Factor}

\author{Eric Naslund}

\date{July 25th 2012}

\email{naslund.eric@gmail.com}

\address{The University of British Columbia Department of Mathematics}
\begin{abstract}
Let $M(x)$ denote the median largest prime factor of the integers
in the interval $[1,x]$. We prove that
\[
M(x)=x^{\frac{1}{\sqrt{e}}\exp\left(-\text{li}_{f}(x)/x\right)}+O_{\epsilon}\left(x^{\frac{1}{\sqrt{e}}}e^{-c\left(\log x\right)^{3/5-\epsilon}}\right),
\]
where $\text{li}_{f}(x)=\int_{2}^{x}\frac{\left\{ x/t\right\} }{\log t}dt$.
From this, we obtain the asymptotic 
\[
M(x)=e^{\frac{\gamma-1}{\sqrt{e}}}x^{\frac{1}{\sqrt{e}}}\left(1+O\left(\frac{1}{\log x}\right)\right),
\]
where $\gamma$ is the Euler Mascheroni constant. This answers a question
posed by Martin \cite{MartinPrimeQuestion}, and improves a result
of Selfridge and Wunderlich \cite{WunderlichSelfridgeMedian}. 
\end{abstract}
\maketitle

\section{Introduction\label{sec:Introduction}}

The median largest prime factor of the integers in the interval $\left[1,x\right]$,
which we denote by $M(x)$, has size 
\begin{equation}
M(x)=x^{1/\sqrt{e}+o(1)}.\label{eq: Selfridge Wunderlich Result}
\end{equation}
This result first appeared in 1974 in a paper by Selfridge and Wunderlich
\cite{WunderlichSelfridgeMedian}. Martin asked \cite{MartinPrimeQuestion}
how the median prime factor compares to $x^{1/\sqrt{e}}$, and whether
we have 
\begin{enumerate}
\item For sufficiently large $x$, $M(x)<x^{1/\sqrt{e}}$.
\item Each inequality $M(x)<x^{1/\sqrt{e}}$ and $M(y)>y^{1/\sqrt{e}}$
holds for arbitrarily large $x,y$.
\item For sufficiently large $x$, $M(x)>x^{\frac{1}{\sqrt{e}}}$. 
\end{enumerate}
In this paper we prove that 
\begin{equation}
M(x)=e^{\frac{\gamma-1}{\sqrt{e}}}x^{\frac{1}{\sqrt{e}}}\left(1+O\left(\frac{1}{\log x}\right)\right),\label{eq:M(x) asymptotic}
\end{equation}
where $\gamma$ is the Euler-Mascheroni constant. As $\exp\left(\frac{\gamma-1}{\sqrt{e}}\right)\approx0.7738$,
it follows that option $1$ holds. Our main theorem is stronger than
this with an error term like that of the prime number theorem, from
which we may obtain an asymptotic expansion for $M(x)$.
\begin{thm}
\label{thm: Median prime factor precise}For every $\epsilon>0$,
\begin{equation}
M(x)=x^{\frac{1}{\sqrt{e}}\exp\left(-\text{li}_{f}(x)/x\right)}+O_{\epsilon}\left(x^{\frac{1}{\sqrt{e}}}e^{-c\left(\log x\right)^{3/5-\epsilon}}\right),\label{eq:M(x) full asymptotic}
\end{equation}
where $\text{li}_{f}(x)=\int_{2}^{x}\frac{\left\{ x/t\right\} }{\log t}dt$,
and $O_{\epsilon}$ means the constant is allowed to depend on $\epsilon$.
\end{thm}
The function $\text{li}_{f}(x)$ admits an asymptotic expansion with
coefficients expressible as sums of the Stieltjes constants.
\begin{prop}
\label{prop: li_f expansion}For any integer $k$, we have the asymptotic
expansion 
\begin{equation}
\text{li}_{f}(x)=c_{0}\frac{x}{\log x}+c_{1}\frac{x}{\log^{2}x}+\cdots+c_{k-1}\frac{(k-1)!x}{\log^{k}x}+O\left(\frac{x}{\log^{k+1}x}\right),\label{eq:li f expansion}
\end{equation}
where $c_{n}=1-\sum_{k=0}^{n}\frac{1}{k!}\gamma_{k}$, and $\gamma_{k}$
denotes the $k^{th}$ Stieltjes.
\end{prop}
Using the Taylor expansion for $e^{z}$ around $z=0$ along with the
above proposition, it follows that 
\[
\exp\left(-\text{li}_{f}(x)/x\right)=1+\frac{1-\gamma}{\log x}+O\left(\frac{1}{\log^{2}x}\right).
\]
Since 
\[
x^{\frac{1}{\sqrt{e}}\frac{1-\gamma}{\log x}}=e^{\frac{\gamma-1}{\sqrt{e}}}x^{\frac{1}{\sqrt{e}}},
\]
we are able to deduce equation \ref{eq:M(x) asymptotic} as a corollary
of theorem \ref{thm: Median prime factor precise}. Applying the same
approach with more terms, it follows that $M(x)$ has an asymptotic
expansion of the form 
\begin{equation}
M(x)=e^{\frac{\gamma-1}{\sqrt{e}}}x^{\frac{1}{\sqrt{e}}}\left(1+\frac{d_{1}}{\log x}+\cdots+\frac{d_{n}}{\log^{n}x}+O_{n}\left(\frac{1}{\log^{n+1}x}\right)\right),\label{eq:M(x) expansion}
\end{equation}
where the $d_{i}$ are computable constants. 

Our proof of theorem \ref{thm: Median prime factor precise} is elementary,
and uses an application of the hyperbola method. In section \ref{sec:The-Mean-of-Omega},
we use theorem \ref{thm: Median prime factor precise} to strengthen
a result of Diaconis \cite{Diaconis1976}, and prove that 
\begin{equation}
\sum_{n\leq x}\omega(n)=x\log\log x+B_{1}x-\text{li}_{f}(x)+O_{\epsilon}\left(xe^{-c\left(\log x\right)^{3/5-\epsilon}}\right).\label{eq:sum of omega(n) diaconis}
\end{equation}
From this, we can recover Diaconis' asymptotic expansion of $\sum_{n\leq x}\omega(n)$
by applying proposition \ref{prop: li_f expansion}. In section \ref{sec:Integers-Without-Large-Prime-Factors},
we use the work of DeBruijn \cite{DeBruijn1951} and Saias \cite{Saias1989}
on integers without large prime factors to give an alternate derivation
of theorem \ref{thm: Median prime factor precise}.

\section{The Main Theorem}

For each prime $p$ greater than $\sqrt{x}$, there is at most one
integer $n\leq x$ such that $p|n$. By the definition of the median
largest prime factor, exactly half of the integers in the interval
$\left[1,x\right]$ will be divisible by a prime $p>M(x)$, and since
$M(x)=x^{1/\sqrt{e}+o(1)}>\sqrt{x}$ there will be no double counting.
It then follows that 
\[
\frac{1}{2}x=\sum_{M(x)<p\leq x}\left[\frac{x}{p}\right]+O(1),
\]
where the $O(1)$ term arises since $x$ may not be an even integer.
We may split up this sum by writing the floor function as $\left[x\right]=x-\left\{ x\right\} $,
where $\left\{ x\right\} $ denotes the fractional part of $x$. Using
Mertens formula 
\begin{equation}
\sum_{p\leq x}\frac{1}{p}=\log\log x+B_{1}+O_{\epsilon}\left(e^{-c\left(\log x\right)^{3/5-\epsilon}}\right),\label{eq:Mertens Formula}
\end{equation}
where $B_{1}=\gamma-\sum_{p}\sum_{k\geq2}\frac{1}{kp^{k}}$, we see
that 
\begin{equation}
\sum_{M(x)<p\leq x}\left\{ \frac{x}{p}\right\} =x\left(\log\frac{\log x}{\log M(x)}-\frac{1}{2}\right)+O(1).\label{eq:M(x) functional equation}
\end{equation}
To understand the left hand side of the equation, we need to find
a precise asymptotic for the sum of the fractional parts $\left\{ x/p\right\} $.
The following proposition strengthens a result of De La Vall\'{e}e
Poussin's \cite{DeLaValleePoussinFractional} where he gave the asymptotic
$\sum_{p\leq x}\left\{ \frac{x}{p}\right\} \sim\frac{1-\gamma}{\log x}$. 
\begin{prop}
\label{prop:frac(x/p) proposition}We have that 
\[
\sum_{p\leq x}\left\{ \frac{x}{p}\right\} =\text{li}_{f}(x)+O_{\epsilon}\left(xe^{-c\left(\log x\right)^{3/5-\epsilon}}\right).
\]
\end{prop}
\begin{proof}
Let $1<B\leq x$ be some integer, and fix $\epsilon>0$. Splitting
into intervals and rearranging we have that 
\begin{eqnarray}
\sum_{\frac{x}{B}<p\leq x}\left[\frac{x}{p}\right] & = & \sum_{n\leq B-1}n\left(\sum_{\frac{x}{n+1}<p\leq\frac{x}{n}}1\right)=\pi(x)+\pi\left(\frac{x}{2}\right)+\cdots+\pi\left(\frac{x}{B-1}\right)-(B-1)\pi\left(\frac{x}{B}\right)\nonumber \\
 & = & \sum_{n\leq B-1}\left(\pi\left(\frac{x}{n}\right)-\pi\left(\frac{x}{B}\right)\right).\label{eq:Hyperbola method}
\end{eqnarray}
By the prime number theorem this is

\begin{eqnarray*}
 & = & \sum_{n\leq B-1}\int_{\frac{x}{B}}^{\frac{x}{n}}\frac{1}{\log t}dt+O\left(\sum_{n\leq B-1}\left(\frac{x}{n}e^{-c\sqrt{\log\frac{x}{n}}}\right)\right)\\
 & = & \int_{\frac{x}{B}}^{x}\frac{\left[x/t\right]}{\log t}dt+O\left(xe^{-c\left(\log\frac{x}{B}\right)^{3/5-\epsilon}}\log B\right).
\end{eqnarray*}
Using the fact that $\left[x\right]=x-\left\{ x\right\} $, the main
term is 
\[
\int_{\frac{x}{B}}^{x}\frac{\left[x/t\right]}{\log t}dt=x\left(\log\log x-\log\log\left(\frac{x}{B}\right)\right)-\int_{\frac{x}{B}}^{x}\frac{\{x/t\}}{\log t}dt,
\]
and hence since $\sum_{\frac{x}{B}\leq p\leq x}\frac{1}{p}=\log\log x-\log\log\left(\frac{x}{B}\right)+O\left(e^{-c\left(\log x\right)^{\frac{3}{5}-\epsilon}}\right)$
by \ref{eq:Mertens Formula}, we have that 
\[
\sum_{\frac{x}{B}<p\leq x}\left\{ \frac{x}{p}\right\} =\int_{\frac{x}{B}}^{x}\frac{\{x/t\}}{\log t}dt+O\left(xe^{-c\left(\log\frac{x}{B}\right)^{3/5-\epsilon}}\log B\right).
\]
The proposition follows by choosing $B=\sqrt{x}$ and noting that
we can extend the sum and integral to start at $2$ as $\int_{2}^{\sqrt{x}}\frac{\{x/t\}}{\log t}dt=O\left(\frac{\sqrt{x}}{\log x}\right)$
and $\sum_{p\leq\sqrt{x}}\left\{ \frac{x}{p}\right\} =O\left(\frac{\sqrt{x}}{\log x}\right).$
\end{proof}
Combining proposition \ref{prop:frac(x/p) proposition} with equation
\ref{eq:M(x) functional equation}, we have 
\begin{equation}
\text{li}_{f}(x)=-x\left(\log\frac{\log M(x)}{\log x}+\frac{1}{2}\right)+O_{\epsilon}\left(xe^{-c\left(\log x\right)^{3/5-\epsilon}}\right),\label{eq:M(x) functional equation with li_f(x)}
\end{equation}
since 
\[
\sum_{M(x)<p\leq x}\left\{ \frac{x}{p}\right\} =\sum_{p\leq x}\left\{ \frac{x}{p}\right\} +O\left(M(x)\right)=\text{li}_{f}(x)+O_{\epsilon}\left(xe^{-c\left(\log x\right)^{3/5-\epsilon}}\right).
\]
Rearranging the equation, we find that 
\[
\frac{\log M(x)}{\log x}=\exp\left(-\frac{1}{2}-\frac{\text{li}_{f}(x)}{x}+O_{\epsilon}\left(e^{-c\left(\log x\right)^{3/5-\epsilon}}\right)\right),
\]
and we are able to turn the error term into an additive factor since
\begin{equation}
\exp\left(O_{\epsilon}\left(e^{-c\left(\log x\right)^{3/5-\epsilon}}\right)\right)=1+O_{\epsilon}\left(e^{-c\left(\log x\right)^{3/5-\epsilon}}\right).\label{eq:error multiplicative to additive}
\end{equation}
It follows that
\[
M(x)=x^{\frac{1}{\sqrt{e}}\exp\left(-\text{li}_{f}(x)/x\right)+O_{\epsilon}\left(\exp\left(-c\left(\log x\right)^{3/5-\epsilon}\right)\right)},
\]
and by using \ref{eq:error multiplicative to additive} again, we
obtain equation \ref{eq:M(x) full asymptotic}, proving theorem \ref{thm: Median prime factor precise}.

\subsection{The Function $\text{li}_{f}(x)$ }

To prove proposition \ref{eq:li f expansion}, we begin by truncating
the interval of integration to obtain 
\begin{equation}
\text{li}_{f}(x)=\int_{\sqrt{x}}^{x}\frac{\{x/t\}}{\log t}dt+O\left(\frac{\sqrt{x}}{\log x}\right).\label{eq:li f proof equation}
\end{equation}
Substituting $u=x/t$, we may write 
\[
\int_{\sqrt{x}}^{x}\frac{\{x/t\}}{\log t}dt=x\int_{1}^{\sqrt{x}}\frac{\{u\}}{u^{2}\log\left(\frac{x}{u}\right)}du=\frac{x}{\log x}\int_{1}^{\sqrt{x}}\frac{\{u\}}{u^{2}}\left(1-\frac{\log u}{\log x}\right)^{-1}du.
\]
Expanding the geometric series
\[
\left(1-\frac{\log u}{\log x}\right)^{-1}=1+\frac{\log u}{\log x}+\cdots+\left(\frac{\log u}{\log x}\right)^{k-1}+\left(\frac{\log u}{\log x}\right)^{k}\left(1-\frac{\log u}{\log x}\right)^{-1},
\]
we see that 
\begin{equation}
\int_{\sqrt{x}}^{x}\frac{\{x/t\}}{\log t}dt=\frac{x}{\log x}\sum_{n=0}^{k-1}\frac{\int_{1}^{\sqrt{x}}\frac{\{u\}}{u^{2}}\left(\log u\right)^{n}du}{\left(\log x\right)^{n}}+\frac{x}{\left(\log x\right)^{k+1}}\int_{1}^{\sqrt{x}}\frac{\{u\}}{u^{2}}\left(\log u\right)^{k}\left(1-\frac{\log u}{\log x}\right)^{-1}du.\label{eq:li_f giant equation}
\end{equation}
The last term contributes an error of the form $O_{k}\left(\frac{x}{\log^{k+1}x}\right)$,
and since we may bound the integral 
\[
\int_{\sqrt{x}}^{\infty}\frac{\{u\}}{u^{2}}\left(\log u\right)^{n}du\leq\int_{\sqrt{x}}^{\infty}\frac{\left(\log u\right)^{n}}{u^{2}}du=O\left(\frac{\left(\log x\right)^{n}}{\sqrt{x}}\right),
\]
it follows that by \ref{eq:li_f giant equation} and \ref{eq:li f proof equation}
we have 
\begin{equation}
\text{li}_{f}(x)=\frac{x}{\log x}\sum_{n=0}^{k-1}\frac{\int_{1}^{\infty}\frac{\{u\}}{u^{2}}\left(\log u\right)^{n}du}{\left(\log x\right)^{n}}+O_{k}\left(\frac{x}{\log^{k+1}x}\right).\label{eq: int proof equation}
\end{equation}
To evaluate the constants explicitly, we will make use of the Laurent
expansion of $\zeta(s)$ which is given by
\begin{equation}
\zeta(s)=\frac{1}{s-1}+\sum_{n=0}^{\infty}\frac{(-1)^{n}}{n!}\gamma_{n}(s-1)^{n},\label{eq:laurent expansion}
\end{equation}
where the $\gamma_{n}$ are the Stieltjes Constants, and $\gamma_{0}$
is the Euler-Mascheroni constant. We will also make use of the identity
\begin{equation}
\zeta(s)=\frac{1}{s-1}+1-s\int_{1}^{\infty}\{x\}x^{-s-1}dx,\label{eq:zeta for s>0}
\end{equation}
which holds for for $s\neq1$, $\text{Re}(s)>0$ \cite{MontVaughn2007}.
Letting 
\begin{equation}
c_{n}=\frac{1}{n!}\int_{1}^{\infty}\frac{\{u\}}{u^{2}}\left(\log u\right)^{n}du,\label{eq:c_n definition}
\end{equation}
we have the following lemma:
\begin{lem}
\label{lem:integral as stieltjes constants}For any integer $n\geq0$,
\[
c_{n}=1-\sum_{k=0}^{n}\frac{1}{k!}\gamma_{n}.
\]
\end{lem}
\begin{proof}
Consider the generating series 
\[
\sum_{n=0}^{\infty}c_{n}z^{n}=\int_{1}^{\infty}\frac{\{u\}}{u^{2}}e^{z\log u}du=\int_{1}^{\infty}\{u\}u^{z-2}du.
\]
By \ref{eq:zeta for s>0}, this equals $\frac{1-\zeta(1-z)-1/z}{1-z}$,
and so from equation \ref{eq:laurent expansion} we have 
\[
\sum_{n=0}^{\infty}c_{n}z^{n}=\left(\sum_{m\geq0}z^{m}\right)\left(1-\sum_{n\geq0}\frac{\gamma_{n}}{n!}z^{n}\right),
\]
and the result follows upon comparing coefficients.
\end{proof}

\section{The Mean of $\omega(n)$ and $\Omega(n)$ \label{sec:The-Mean-of-Omega}}

Using proposition \ref{prop:frac(x/p) proposition}, we are able to
provide a short proof of the asymptotic expansion of $\sum_{n\leq x}\omega(n)$
given in \cite{Diaconis1976}, where $\omega(n)$ is the number of
distinct prime divisors function. Since $\omega(n)=\sum_{p|n}1$,
rearranging the orders of summation implies that 
\[
\sum_{n\leq x}\omega(n)=\sum_{n\leq x}\sum_{p|n}1=\sum_{p\leq x}\left[\frac{x}{p}\right]=x\sum_{p\leq x}\frac{1}{p}-\sum_{p\leq x}\left\{ \frac{x}{p}\right\} .
\]
By \ref{eq:Mertens Formula} and proposition \ref{prop:frac(x/p) proposition},
we deduce equation \ref{eq:sum of omega(n) diaconis} which states
that 
\[
\sum_{n\leq x}\omega(n)=x\log\log x+B_{1}x-\text{li}_{f}(x)+O_{\epsilon}\left(xe^{-c\left(\log x\right)^{3/5-\epsilon}}\right).
\]
Combining the previous equation with proposition \ref{prop: li_f expansion}
yields Diaconis' \cite{Diaconis1976} expansion 
\begin{equation}
\sum_{n\leq x}\omega(n)=x\log\log x+B_{1}x-\frac{c_{0}x}{\log x}+\cdots+\frac{c_{k-1}(k-1)!x}{\log^{k}x}+O\left(\frac{x}{\log^{k+1}x}\right),\label{eq:Diaconis expansion}
\end{equation}
with $c_{n}$ given explicitely as $c_{n}=1-\sum_{k=0}^{n}\frac{1}{k!}\gamma_{k}$.
We may derive a similar expansion for $\Omega(n)$, the number of
distinct prime factors counted with multiplicity. Taking into account
the higher prime powers, we see that 
\[
\sum_{n\leq x}\Omega(n)-\omega(n)=x\sum_{p}\frac{1}{p(p-1)}+O\left(\sqrt{x}\log x\right),
\]
and so 
\begin{equation}
\sum_{n\leq x}\Omega(n)=x\log\log x+B_{2}x-\text{li}_{f}(x)+O\left(xe^{-c\left(\log x\right)^{3/5-\epsilon}}\right)\label{eq:Omega(n) asymptotic}
\end{equation}
where $B_{2}=B_{1}+\sum_{p}\frac{1}{p(p-1)}.$

\section{Integers Without Large Prime Factors \label{sec:Integers-Without-Large-Prime-Factors}}

In this section, we deduce the main result in a different way using
the work of DeBruijn \cite{DeBruijn1951} and Saias \cite{Saias1989}
on integers without large prime factors. Let $\psi(x,y)$ denote the
number of integers $n$ with $1\leq n\leq x$, all of whose prime
factors are $\leq y$. Then the median largest prime factor, $M(x)$,
of the integers in the interval $[1,x]$ satisfies 
\[
\psi\left(x,M(x)\right)=\frac{1}{2}x+O(1).
\]
In \cite{DeBruijn1951}, De Bruijn showed that 
\[
\psi(x,y)=\Lambda(x,y)+O_{\epsilon}\left(xe^{-c\left(\log x\right)^{3/5-\epsilon}}\right),
\]
where 
\[
\Lambda(x,y)=x\int_{0}^{x}\rho\left(\frac{\log x-\log t}{\log y}\right)d\frac{\left[t\right]}{t},
\]
and $\rho(u)$ denotes the Dickmann De Bruijn rho function. It follows
that we are looking for $y$ such that 
\begin{equation}
\Lambda(x,y)=\frac{1}{2}x+O_{\epsilon}\left(xe^{-c\left(\log x\right)^{3/5-\epsilon}}\right).\label{eq:Lambda(x,y) equation}
\end{equation}
Examining $\Lambda(x,y)$ more closely, integration by parts yields
\[
\int_{0}^{x}\rho\left(\frac{\log x-\log t}{\log y}\right)d\frac{\left[t\right]}{t}=1+\frac{1}{\log y}\int_{0}^{x}\frac{\left[t\right]}{t^{2}}\rho^{'}\left(\frac{\log x-\log t}{\log y}\right)dt.
\]
Substituting $s=\frac{x}{t}$, and using the fact that $\rho^{'}(u)=0$
when $0<u<1$, we have 
\[
\Lambda(x,y)=x+\frac{1}{\log y}\int_{y}^{x}\left[x/s\right]\rho^{'}\left(\frac{\log s}{\log y}\right)ds.
\]
In our case, since $x>y>\sqrt{x}$, we are on the interval $1<u<2$,
and on this range $\rho(u)=1-\log u$ so that $\rho'(u)=-\frac{1}{u}$.
Thus we have that for $x\geq y>\sqrt{x}$, 
\[
\Lambda(x,y)=1-\int_{y}^{x}\frac{\left[x/s\right]}{\log s}ds,
\]
and by splitting up the floor function and recalling the definition
of $\text{li}_{f}(x)$, it follows that 
\begin{eqnarray}
\Lambda(x,y) & = & x-x\int_{y}^{x}\frac{1}{s\log s}ds+\int_{y}^{x}\frac{\left\{ x/s\right\} }{\log s}ds\nonumber \\
 & = & x\left(1-\log\left(\frac{\log x}{\log y}\right)\right)+\text{li}_{f}(x)+O\left(y\right).\label{eq:lambda(x,y) and lif equivalence}
\end{eqnarray}
With \ref{eq:lambda(x,y) and lif equivalence} in hand, solving equation
\ref{eq:Lambda(x,y) equation} for $M(x)$ becomes identical to solving
equation \ref{eq:M(x) functional equation with li_f(x)}, and so this
gives a way to deduce theorem \ref{thm: Median prime factor precise}
from De Bruijn's work. To obtain proposition \ref{prop: li_f expansion},
we use the expansion for $\Lambda(x,y)$ given in Saias' paper \cite{Saias1989}.
Suppose that $x^{\frac{1}{u}}=y$, $u\leq\left(\log y\right)^{\frac{3}{5}-\epsilon}$,
and that $u\in\cup_{1\leq k\leq n}\left(k+\epsilon,k+1\right)\cup\left(n+1,\infty\right)$,
so that $u$ is not too close to an integer. Then then we have the
expansion 
\begin{equation}
\Lambda(x,y)=x\sum_{k=0}^{n}a_{k}\frac{\rho^{(k)}(u)}{\left(\log y\right)^{k}}+O_{n,\epsilon}\left(x\frac{\rho^{(n+1)}(u)}{\left(\log y\right)^{n+1}}\right),\label{eq:Eric Saias result}
\end{equation}
where 
\[
a_{0}=1,\ \ \ \ a_{k}=\frac{(-1)^{k}}{(k-1)!}\int_{1}^{\infty}\frac{\left\{ t\right\} }{t^{2}}\left(\log t\right)^{k-1}dt.
\]
In our case, $1<u<2$, and on this range $\rho(u)=1-\log u$ so that
$\rho'(u)=-\frac{1}{u}$ and $\rho^{\left(k\right)}(u)=\frac{(k-1)!(-1)^{k}}{u^{k}}.$
We note that, by \ref{lem:integral as stieltjes constants}, for $k\geq1$
\[
a_{k}=(-1)^{k}c_{k-1}=(-1)^{k}\left(1-\sum_{j=0}^{k-1}\frac{1}{j!}\gamma_{j}\right).
\]
Since $\frac{1}{\left(\log y\right)^{k}}=\frac{u^{k}}{\left(\log x\right)^{k}}$,
equation \ref{eq:Eric Saias result} becomes 
\begin{equation}
\Lambda(x,y)=x\rho(u)+\sum_{k=1}^{n}\frac{(k-1)!c_{k-1}}{\left(\log x\right)^{k}}+O_{n}\left(\frac{x}{\left(\log x\right)^{n+1}}\right).\label{eq:Eric Saias rewritten nicely}
\end{equation}
Combining equation \ref{eq:Eric Saias rewritten nicely} with the
fact that $\rho(u)=1-\log\left(\frac{\log x}{\log y}\right)$ when
$y>\sqrt{x}$, we are able to obtain the expansion in proposition
\ref{prop: li_f expansion} from equation \ref{eq:lambda(x,y) and lif equivalence}.
We note that this implies that equations \ref{eq:sum of omega(n) diaconis}
and \ref{eq:Diaconis expansion} on the mean of $\omega(n)$ follow
from the work of De Bruijn and Saias on integers without large prime
factors.

\specialsection*{Acknowledgments}

I would like to thank Qiaochu Yuan for providing the proof of Lemma
\ref{lem:integral as stieltjes constants} on Math Overflow. I would
also like to thank Andrew Granville for his insightful comments, and
Kevin Ford for his helpful advice. I am indebted to Greg Martin for
both his mathematical suggestions and his encouragement, as this project
would not have started or progressed otherwise.

\bibliographystyle{plain}

\end{document}